\documentclass[reqno,12pt,a4paper]{amsart}

\usepackage[utf8]{inputenc}

\usepackage{enumerate}
\usepackage{amstext,amsmath,amsthm,amsfonts,amssymb,amscd}
\usepackage{latexsym,mathrsfs,dsfont,euscript}
\usepackage{pict2e}
\usepackage{fullpage}
\usepackage{xcolor}

\usepackage{hyperref} 
\hypersetup{
  colorlinks,%
  citecolor=blue,%
    filecolor=red,%
    linkcolor=red,%
    urlcolor=blue
}

\newtheorem{theorem}{Theorem}[section]
\newtheorem{proposition}[theorem]{Proposition}
\newtheorem{lemma}[theorem]{Lemma}

\theoremstyle{definition}

\theoremstyle{remark}

\numberwithin{equation}{section}

\newcommand{\abs}[1]{\left\vert#1\right\vert}
\newcommand{\set}[1]{\left\{#1\right\}}
\newcommand{\proin}[2]{\left<#1,#2\right>}
\newcommand{\norm}[1]{\left\Vert#1\right\Vert}

\makeatletter
\DeclareRobustCommand{\loplus}{\mathbin{\mathpalette\dog@lsemi{+}}}
\DeclareRobustCommand{\lotimes}{\mathbin{\mathpalette\dog@lsemi{\times}}}
\DeclareRobustCommand{\roplus}{\mathbin{\mathpalette\dog@rsemi{+}}}
\DeclareRobustCommand{\rotimes}{\mathbin{\mathpalette\dog@rsemi{\times}}}

\newcommand{\dog@rsemi}[2]{\dog@semi{#1}{#2}{-90,90}}
\newcommand{\dog@lsemi}[2]{\dog@semi{#1}{#2}{270,90}}
\newcommand{\dog@semi}[3]{%
  \begingroup
  \sbox\z@{$\m@th#1#2$}%
  \setlength{\unitlength}{\dimexpr\ht\z@+\dp\z@\relax}%
  \makebox[\wd\z@]{\raisebox{-\dp\z@}{%
    \begin{picture}(1,1)
    \linethickness{\variable@rule{#1}}
    \roundcap
    \put(0.5,0.5){\makebox(0,0){\raisebox{\dp\z@}{$\m@th#1#2$}}}
    \put(0.5,0.5){\arc[#3]{0.7}}
    \end{picture}%
  }}%
  \endgroup
}
\newcommand{\variable@rule}[1]{%
  \fontdimen8
  \ifx#1\displaystyle\textfont3\else
    \ifx#1\textstyle\textfont3\else
      \ifx#1\scriptstyle\scriptfont3\else
        \scriptscriptfont3\relax
  \fi\fi\fi
}
\makeatother

\allowdisplaybreaks
\begin{document}
\title[]{Kirchhoff divergence and diffusions associated to transport probability measures}
%


\author[]{Hugo Aimar}
\email{haimar@santafe-conicet.gov.ar}
\author[]{Ivana G\'{o}mez}
\email{ivanagomez@santafe-conicet.gov.ar}
\thanks{This work was supported by the Ministerio de Ciencia, Tecnolog\'ia e Innovaci\'on de la Naci\'on-MINCYT in Argentina: CONICET and ANPCyT; and UNL}



\begin{abstract}
	Taking as starting point the approach to the divergence operator on weighted graphs, we give a notion of divergence associated to transport coupling and coupled measures on locally compact Hausdorff spaces. We consider the induced Laplacian operator and the corresponding heat diffusion operators in some particular instances.
\end{abstract}
\maketitle

\section{Introduction}

In the classical approach in the $n$-dimensional Euclidean space $\mathbb{R}^n$ as the iteration of the operator nabla, $\Delta= div\,grad\,=\nabla^2$, the Laplacian, can be seen as the composition of two differential operators each of first order of differentiation. The gradient acting on scalars and the divergence on vectors. The intrinsic definitions of these operators are both based in the geometry of the Euclidean space. In fact, the gradient of a scalar function $u$ at a point $P$ in $\mathbb{R}^n$, is the vector with the direction of the maximal growth of $u$ from $P$ and length equal to the rate of growth of $u$ in that maximal direction. The divergence of a vector field $F$ at a point $P$ of the space is given by the outer flow of $F$, from $P$, per unit volume. Sometimes some problems of potential analysis are posed on sets without any a priori geometric or algebraic structure and there is no other way to measure a gradient of a potential $u$, than the mere difference $u(y)-u(x)$ for any two different points $x$ and $y$ of the domain of $u$. A classical instance of this situation is provided by the Kirchhoff laws in the theory of electric circuits. A resistive circuit of $n$ nodes, $\set{1,2,\ldots, n}$ can be schematically seen as a weighted graph. Assume that $R_{ij}$ is the electrical resistance of the connection between nodes $i$ and $j$ of the circuit. If we admit for each $R_{ij}$ any nonnegative value  including $+\infty$, and $\Phi_{ij}$ is the potential difference between nodes $i$ and $j$, from Ohm's law, the sum of all current intensities is given at each node $i=1,2,\ldots,n$, by $(K\Phi)_i=\sum_{j=1}^n \frac{\Phi_{ij}}{R_{ij}}=\sum_{j=1}^n w_{ij}\Phi_{ij}$, where $w_{ij}=\frac{1}{R_{ij}}$. A function $\Phi$ defined on the edges of the graph will satisfy the first Kirchhoff law if $K\Phi=0$. The first Kirchhoff law can be seen as a conservation law in the sense that at each node of the electrical circuit all incoming electrical currents have to compensate with all out-coming electrical currents. In this sense the operator $K$ acting on functions $\Phi$ defined on the edges of the graph, can be seen as a divergence operator. When $\Phi$ is $\nabla u$, the ``naive'' gradient of $u$, $(\nabla u)_{ij}=u_j-u_i$ and $u$ is any function defined on the nodes of the graph, we have the Laplacian operator $\Delta = K\nabla$. In particular, $u$ is harmonic ($\Delta u=0$) if and only if $u$ satisfies the mean value identity at each node, i.e. $u_i =\sum_{j=1}^n w_{ij}u_j$, whose global or local characters depend on the concentration of the nonvanishing terms of the matrix $w_{ij}$. For a complete analysis see \cite{DoyleSnellBook84}.

Let us start by some basic abstract definitions. In the following sections we shall profusely exemplify and illustrate the general framework.

Let $X$ be a locally compact Hausdorff space. Set $\mathscr{C}_c(X)$ and $\mathscr{C}_c(X\times X)$ to denote the spaces of real  valued continuous functions with compact support, defined on $X$ and $X\times X$ respectively. We shall use capital Greek letters $\Phi=\Phi(x,y)$ to denote real functions defined on $X\times X$ and small Greek letters $\varphi=\varphi(x)$ to denote real functions defined on $X$. Let $\mu$ be a Borel probability measure on $X$ and $\pi$ be a Borel probability measure on $X\times X$. For a given $\Phi\in\mathscr{C}_c(X\times X)$ we say that a function $\psi$ defined on $X$ is a \textbf{Kirchhoff divergence of $\Phi$ with respect to $\mu$ and $\pi$} if the equation
\begin{equation}\label{eq:KichDivOperator}
\int_X \varphi(x)\psi(x) d\mu(x) = \iint_{X\times X} \varphi(x)\Phi(x,y) d\pi(x,y)
\end{equation}
holds for every $\varphi$ in $\mathscr{C}_c(X)$. It is not difficult to provide examples showing that solutions of \eqref{eq:KichDivOperator} may not exist and may not be unique. For nonexistence take $X=[0,1]$, $\mu=\delta_0$ the Dirac delta at the origin, $d\pi= dx dy$ the area in the unit square and $\Phi\equiv 1$. To show a case of non-uniqueness take again $X=[0,1]$, $\mu=\delta_0$, $\Phi\equiv 1$ and $\pi=\delta_0\times \delta_0$.

Given a finite measure $\pi$ on $X\times X$, as usual, we denote with $\pi^i$, $i=1,2$; the two marginal distributions of $\pi$, for $B$ a Borel set in $X$, $\pi^1(B)=\pi(B\times X)$ and $\pi^2(B)=\pi(X\times B)$.

Radon-Nikodym Theorem provides a simple criteria for existence and uniqueness, up to null sets, of a solution $\psi$ of \eqref{eq:KichDivOperator} which shall be enough for our further work. Given $\Phi\in\mathscr{C}_c(X\times X)$ and $\pi$ a probability on $X\times X$, we write $\pi_\Phi$ to denote the measure $d\pi_\Phi(x,y)=\Phi(x,y)d\pi(x,y)$ and $\pi^1_\Phi$ and $\pi^2_\Phi$ to denote the first and second marginals of $\pi_\Phi$.
\begin{proposition}
	Let $X$ be a locally compact Hausdorff space. Let $\mu$ be a Borel probability on $X$, $\pi$ a Borel probability on $X\times X$ and $\Phi\in\mathscr{C}_c(X\times X)$. If the first marginal $\pi^1_\Phi$ of $\pi_\Phi$ is absolutely continuous with respect to $\mu$, then, the Radon-Nikodym derivative of $\pi^1_\Phi$ with respect to $\mu$, $\frac{d\pi^1_\Phi}{d\mu}$, solves \eqref{eq:KichDivOperator}.
\end{proposition}
\begin{proof}
Notice that since $\pi_{\Phi}^1$ is the first marginal of $\pi_{\Phi}$, we have that $\int_X\varphi(x) d\pi_{\Phi}^1(x)= \iint_{X\times X}\varphi(x) d\pi_{\Phi}(x,y)=\iint_{X\times X}\varphi(x)\Phi(x,y) d\pi(x,y)$ for every $\varphi\in\mathscr{C}_c(X)$. On the other hand, since we are assuming $\pi_{\Phi}^1<<\mu$, we have
\begin{equation*}
\iint_{X\times X}\varphi(x)\Phi(x,y) d\pi(x,y) = \int_X\varphi(x)d\pi^1_\Phi(x) = \int_X\varphi(x)\frac{d\pi_{\Phi}^1(x)}{d\mu} d\mu(x)
\end{equation*}
for every $\varphi\in\mathscr{C}_c(X)$. Which gives \eqref{eq:KichDivOperator} with $\psi=\frac{d\pi_{\Phi}^1}{d\mu}$, as desired.
\end{proof}

A probability measure $\pi$ on $X\times X$ is said to be a coupling (see \cite{Villanibook}) between the probability measures $\mu$ and $\nu$ on $X$ if $\pi^1=\mu$ and $\pi^2=\nu$. If $\pi$ is a given probability on $X\times X$, then $\pi^1_\Phi$ is absolutely continuous with respect to $\pi^1=\mu$ and also $\pi^2_\Phi$ is absolutely continuous with respect to $\pi^2=\nu$, no matter what is the particular $\Phi\in\mathscr{C}_c(X\times X)$.

\begin{proposition}
	Let $\mu$ and $\nu$ be two probability measures on the Borel sets of $X$ and let $\pi$ be a coupling of $\mu$ and $\nu$. Then for every $\Phi\in\mathscr{C}_c(X\times X)$, the measures $\pi^1_\Phi$ and $\pi^2_\Phi$ are absolutely continuous with respect to $\mu$ and $\nu$ respectively.
\end{proposition}
\begin{proof}
	Since $\Phi$ is bounded, we have that for every Borel set $B$ in $X$,
\begin{align*}
	\abs{\pi^1_\Phi (B)} & = \abs{\pi_\Phi (B\times X)}\\
	& = \abs{\iint_{B\times X}\Phi(x,y) d\pi(x,y)}\\
	&\leq \norm{\Phi}_\infty \pi^1_\Phi(B)\\
	&= \norm{\Phi}_\infty \mu(B).
	\end{align*}
Also, $\abs{\pi^2_\Phi (B)}\leq \norm{\Phi}_\infty \nu(B).$
\end{proof}
The above propositions prove the following statement.
\begin{theorem}\label{thm:RadonNikodymFirstMarginal}
	Let $X$ be a locally compact Hausdorff space. Let $\pi$ be a Borel probability measure on $X\times X$. Let $\Phi\in\mathscr{C}_c(X\times X)$. Then the Radon-Nikodym derivative $\frac{d\pi^1_\Phi}{d\mu}$ solves equation \eqref{eq:KichDivOperator} with $\mu=\pi^1$, the first marginal of $\pi$, and $\pi^1_\Phi$ is the first marginal of $\pi_{\Phi}(A)=\iint_A \Phi d\pi$, $A$ Borel set in $X\times X$.
\end{theorem}
We shall use the notation $Kir_\pi \Phi$ for the solution $\frac{d\pi^1_\Phi}{d\mu}$
 of \eqref{eq:KichDivOperator} in this case. Notice that $\Phi$ continuous and bounded suffice for the above results. In particular, if $f:X\to \mathbb{R}$ is continuous and bounded, so is $F(x,y)=f(y)-f(x)$, and we can define a Laplacian type operator based on Kirchhoff divergence by
 \begin{equation*}
 \Delta_\pi f = Kir_\pi F.
 \end{equation*}
 Hence, at least formally, the solution for the heat conduction problem
 \begin{equation*}
 (P)\,\left\{
 \begin{array}{ll}
 \frac{\partial u}{\partial t} = \Delta_\pi u,\, & t>0, x\in X\\
 u(x,0)=g(x) \, & x\in X
 \end{array}
 \right.
 \end{equation*}
 is given by
 \begin{equation*}
 u(x,t) = (e^{t \Delta_\pi}g)(x).
 \end{equation*}
 Or, when a spectral resolution of the operator $\Delta_\pi$, in terms of a sequence of eigenvalues $\lambda_i$ and eigenfunctions $\psi_i$, is available
 \begin{equation*}
 u(x,t) = \sum_i e^{t \lambda_i} \proin{g}{\psi_i} \psi_i (x)
 \end{equation*}
 with $\proin{g}{\psi_i} = \int_X g(x)\psi_i (x) d\mu(x)$.

 Along the next sections we aim to explore the above abstract setting for some particular couplings.

\section{The finite case. Weighted graphs}\label{sec:FiniteCase}

Let $X=\mathcal{V}=\set{1,2,\ldots,n}$ be the set of vertices of a weighted graph $\mathcal{G}=(\mathcal{V},E,w)$, where $E=\set{(i,j): i\in\mathcal{V}, j\in\mathcal{V}}=X\times X$ is the set of all edges of $\mathcal{G}$, and $w:E\to\mathbb{R}^+\cup\{0\}$ is the nonnegative weight of each edge, with $w_{ij}=w_{ji}$, $w_{ii}=0$ for every $i$, $w_{ij}>0$, $i\neq j$ and $\sum_{i=1}^n\sum_{j=1}^n w_{ij}=1$. The weight $w$ determines the probability measure $\pi$ by $\pi(A)=\sum_{(i,j)\in A}w_{ij}$ for a subset $A$ of $E$. Hence the first marginal $\mu$ of $\pi$ is given by the weights $\mu_i=\sum_{j=1}^n w_{ij}$. In other words, $\mu(B)=\sum_{i\in B}\mu_i=\sum_{i\in B}\sum_{j=1}^n w_{ij}=\pi(B\times X)$. So that
\begin{equation*}
\int_X \varphi d\mu = \sum_{i=1}^{n} \mu_i \varphi_i, \textrm{ and }
\end{equation*}
\begin{equation*}
\iint_{X\times X} \Phi d\pi =\sum_{j=1}^n\sum_{i=1}^n w_{ij}\Phi_{ij}.
\end{equation*}
Equation \eqref{eq:KichDivOperator} takes the form
\begin{equation*}
\sum_{i=1}^n \mu_i\varphi_i \psi_i =\int_X \varphi \psi d\mu=
\iint_{X\times X} \varphi(x)\Phi(x,y) d\pi =\sum_{j=1}^n\sum_{i=1}^n w_{ij}\varphi_i\Phi_{ij} = \sum_{i=1}^n\varphi_i \left(\sum_{j=1}^n w_{ij}\Phi_{ij}\right),
\end{equation*}
which should hold for every $\varphi$. Hence
$\mu_i\psi_i = \sum_{j=1}^n w_{ij}\Phi_{ij}$,
or, since each $\mu_i$ is positive
\begin{equation*}
\psi_i =(Kir\, \Phi)_i =\frac{1}{\mu_i}\sum_{j=1}^n w_{ij}\Phi_{ij}=\frac{1}{\sum_{i=1}^n w_{ij}}\sum_{j=1}^n w_{ij}\Phi_{ij}.
\end{equation*}
And the corresponding Laplace's operator of a function $f$ defined on the vertices, is given by
\begin{equation*}
(\Delta_\pi f)_i =\frac{1}{\mu_i}\sum_{j=1}^n w_{ij}(f_j - f_i).
\end{equation*}
The harmonic functions in this setting are those that satisfy the mean value identity
\begin{equation*}
f_i = \frac{1}{\sum_{j=1}^n w_{ij}}\sum_{j=1}^n w_{ij} f_j.
\end{equation*}
In matrix notation, $\Delta_\pi=D^{-1}W-I$, where $D = \textrm{diagonal}(\mu_1,\ldots,\mu_n)$, and $W=(w_{ij})$. The diffusion problem
\begin{equation*}
\left\{
\begin{array}{ll}
\frac{\partial u}{\partial t} = \Delta_\pi u,\, & t>0\\
u(i,0)=f_i \, &i=1,\ldots,n
\end{array}
\right.
\end{equation*}
has the solution $\overline{u}(t)=e^{t\Delta_\pi}\overline{f}$, $\overline{f}=(f_1,\ldots,f_n)$.

We have $e^{t\Delta_\pi}=e^{-t}e^{tD^{-1}W}$.
The general theory of Markov chains can be applied to the analysis of the steady state for the solution $\overline{u}(t)$ of $(P)$. See \cite{Roberts1976discrete}. A $n\times n$ transition matrix $B$ is said to correspond to a regular Markov chain, if some power of $B$ has only positive elements. The Fundamental Limit Theorem for regular Markov chains proves that there exist a Markov matrix $M$ with all rows equal to $m=(m_1,\ldots,m_n)$, with $m_i>0$ for every $i=1, \ldots, n$ and $\sum_{i=1}^n m_i=1$, such that
\begin{equation*}
\lim_{k\to\infty} B^k = M.
\end{equation*}
The next result is a particular instance of convergence to equilibrium (see \cite{NorrisBook97}).
\begin{proposition}\label{propo:limitMarkovmatrix}
For $n\geq 3$, let $W=(w_{ij})$ be a nonnegative $n\times n$ matrix such that $w_{ii}=0$ for every $i=1,\ldots,n$, $w_{ij}=w_{ji}>0$ for each $i\neq j$ and $\sum_{i=1}^n\sum_{j=1}^n w_{ij}=1$. Let $\pi$ be the probability measure defined on $\{1,\ldots,n\}^2$ by $\pi(A)=\sum_{(i,j)\in A} w_{ij}$. Given a function $\overline{f}=(f_1,\ldots,f_n)$ defined on the vertices $\mathcal{V}$ of the weighted graph $\mathcal{G}$, set $\overline{u}(t)= e^{t\Delta_\pi}\overline{f}$ to denote the solution of $(P)$ with $t>0$. Then
\begin{equation*}
\lim_{t\to\infty} \overline{u}(t) = \left(\sum_{j=1}^n m_j f_j\right)\overline{1},
\end{equation*}
where $\overline{1}=(1,\ldots,1)$ and $\overline{m}=(m_1,\ldots,m_n)$ is the constant row of the Markov limit matrix
\begin{equation*}
(D^{-1}W)^{\infty} = \lim_{k\to\infty} (D^{-1}W)^k = \begin{pmatrix}
\overline{m}\\
\vdots\\
\overline{m}
\end{pmatrix}
\end{equation*}
\end{proposition}

\begin{proof}
Notice first that $\widetilde{W}=D^{-1}W$ is the matrix of a Markov chain positive entries except for its diagonal terms. Hence $\widetilde{W}$ is a regular Markov chain. From the Fundamental Limit Theorem for regular Markov chains, we have that
\begin{equation*}
\lim_{k\to\infty} \widetilde{W}^k = M = \begin{pmatrix}
\overline{m}\\
\vdots\\
\overline{m}
\end{pmatrix}
\end{equation*}
with $\overline{m}=(m_1,\ldots,m_n)$, $m_i>0$ for every $i=1,\ldots,n$ and $\sum_{i=1}^n m_i =1$. On the other hand,
\begin{equation*}
\overline{u}(t) = e^{-t}\sum_{l\geq 0}\frac{t^l}{l!}\widetilde{W}^l\overline{f},
\end{equation*}
for $t>0$. With the standard notation for norms in $\mathbb{R}^n$, we have that for every $\varepsilon>0$ there exists an integer $L$ such that for $l>L$ we have $\norm{\widetilde{W}^l\overline{f}-M\overline{f}}<\frac{\varepsilon}{2}$. Hence
\begin{align*}
\norm{\overline{u}(t)-\left(\sum_{j=1}^n m_j f_j\right)\overline{1}}
& = \norm{\overline{u}(t)-M\overline{f}}\\
&= \norm{e^{-t}\sum_{l\geq 0}\frac{t^l}{l!}(\widetilde{W}^l\overline{f}-M\overline{f})}\\
& \leq e^{-t}\sum_{l=0}^L\frac{t^l}{l!}\norm{\widetilde{W}^l\overline{f}-M\overline{f}} +
\frac{\varepsilon}{2}e^{-t}\sum_{l\geq L+1}\frac{t^l}{l!}\\
&<\varepsilon
\end{align*}
for $t$ large enough.
\end{proof}

\section{Markov coupling}
Let $X$ be a locally compact Hausdorff space. Let $\mu$ and $\nu$ be two probabilities on the Borel subsets of $X$. Let $\pi$ be a probability on the Borel sets of $X\times X$ that is absolutely continuous with respect to $\mu\times \nu$.
\begin{lemma}\label{lem:couplingMarkov}
Let $X$, $\mu$, $\nu$ and $\pi$ be as described. Then $\pi$ is a coupling for $\mu$ and $\nu$ if and only if $K=\frac{d\pi}{d(\mu\times\nu)}$ is a Markov kernel in the sense that
\begin{enumerate}[(i)]
\item $\int_X K(x,y) d\mu(x)=1$ for $\nu$ almost every $y\in X$, and
\item $\int_X K(x,y) d\nu(y)=1$ for $\mu$ almost every $x\in X$.
\end{enumerate}
\end{lemma}
\begin{proof}
Assume first that $\pi$ is a coupling for $\mu$ and $\nu$. Then for every $B$, Borel set in $X$, we have
\begin{equation*}
\mu(B)=\pi(B\times X)=\iint_{B\times X} \frac{d\pi}{d(\mu\times \nu)} d\mu d\nu
=\int_B\left(\int_X K(x,y) d\nu(y)\right) d\mu(x).
\end{equation*}
Hence $\int_X K(x,y) d\nu(y) =1$ for $\mu$ almost every $x\in X$. Identity $(i)$ follows the same argument. Assume now that $(i)$ and $(ii)$ hold, then by Fubini's theorem
\begin{equation*}
\pi(B\times X)=\iint_{B\times X} d\pi=\iint_{B\times X} K(x,y) d\mu(x) d\nu(y)=\int_B\left(K(x,y) d\nu(y)\right) d\mu(x)=\mu(B).
\end{equation*}
\end{proof}

The next statement provides the Kirchhoff divergence operator associated to these type of Markov couplings.
\begin{theorem}
	Let $X$ be a locally compact Hausdorff topological space. Let $\mu$ and $\nu$ be two given Borel probability measures on $X$. Let $\pi$ be the coupling for $\mu$ and $\nu$ given by
	\begin{equation*}
	\pi(A) = \iint_A K(x,y) d\mu(x) d\nu(y)
	\end{equation*}
	with $K$ satisfying $(i)$ and $(ii)$ in Lemma~\ref{lem:couplingMarkov} and $A$ any Borel set in $X\times X$. Then, for $\Phi\in\mathscr{C}_c(X\times X)$ we have
	\begin{equation*}
	Kir_\pi \Phi(x) = \int_{y\in X} K(x,y) \Phi(x,y) d\nu(y).
	\end{equation*}
\end{theorem}

\begin{proof}
	The measure $\pi_{\Phi}$ induced by $\Phi\in\mathscr{C}_c(X\times X)$ is now given by $\pi_{\Phi}(A)=\iint_A \Phi K d\mu d\nu$. Its first marginal is\begin{equation*}
	\pi_{\Phi}^1(B) = \pi_{\Phi}(B\times X)=\int_B\left(\int_{y\in X}\Phi(x,y) K(x,y) d\nu(y)\right) d\mu(x).
	\end{equation*}
	Hence
	\begin{equation*}
	Kir_\pi \Phi(x)=\frac{d\pi_{\Phi}^1}{d\mu} = \int_{y\in X}\Phi(x,y)K(x,y) d\nu(y),
	\end{equation*}
	as desired.
\end{proof}

For a function $f\in\mathscr{C}_c(X)$, the function $\Phi(x,y)=f(y)-f(x)$ is continuous and bounded and the operator $Kir_\pi$ is well defined on $\Phi$. The Laplacian of $f$ in this setting is, then
\begin{equation*}
\Delta_\pi f(x) = \int_{y\in X} (f(y)-f(x)) K(x,y) d\nu(y) = \int_{y\in X} K(x,y) f(y) d\nu(y) - f(x).
\end{equation*}
Or, in terms of operators, $\Delta_\pi=\mathcal{K}-I$, with $\mathcal{K}f(x)=\int_{y\in X} K(x,y)f(y) d\nu(y)$.
In this setting, the corresponding diffusion (P), is given by
$u(x,t) =e^{t\Delta_\pi} f(x)$,
with $e^{t\Delta_\pi}=\sum_{k\geq 0}\frac{t^k}{k!}(\mathcal{K}-I)^k= e^{-t}\sum_{m\geq 0}\frac{t^m}{m!}\mathcal{K}^m$.

As before, the existence and structure of steady states for $t\to\infty$ depends on the particular dynamics of the sequence $\{\mathcal{K}^j: j\geq 0\}$ of iterations of the operator $\mathcal{K}$ induced by the kernel $K$ and the measure $\nu$ on $X$. Next we explore the case of dyadic Markov kernels, where the spectral analysis can be explicitly carried through Haar type systems built on dyadic type families on abstract settings.

Let $X=[0,1)$ with the usual distance. Let $\mathcal{D}=\cup_{j\geq 0}\mathcal{D}^j$, $\mathcal{D}^j=\{I^j_k=[k2^{-j},(k+1)2^{-j}): k=0,1,\ldots,2^j-1\}$ be the family of standard dyadic intervals in $[0,1)$. For $j\geq 1$, $I=I^j_k\in\mathcal{D}^j$ the Haar function $h_I$ is given by $h_I(x)=2^{j/2}h^0_0(2^j x-k)$, with $h^0_0(x)=\mathcal{X}_{[0,\tfrac{1}{2})}(x)-\mathcal{X}_{[\tfrac{1}{2},1)}(x)$. The system $\mathscr{H}\cup \{\mathcal{X}_{[0,1)}\}$, with $\mathscr{H}=\{h_I: I\in\mathcal{D}\}$, provides an orthonormal basis for $L^2([0,1), dx)$. The dyadic family $\mathcal{D}$ provides also a natural metric structure on $[0,1)$. For $x$ and $y$ in $[0,1)$, set $\delta(x,y)=\inf\{\abs{I}: I\in\mathcal{D} \textrm{ with } x,y\in I\}$. Then $\delta$ is an ultrametric on $[0,1)$ whose balls are the dyadic intervals. In fact, $B_{\delta}(x,r)=\{y\in [0,1): \delta(x,y)<r\}=[0,1)$ for $r\geq 1$ and $B_{\delta}(x,r)=I$, where $x\in I\in\mathcal{D}^j$ and $2^{-j}<r\leq 2^{-j+1}$.

Let $d\mu=d\nu=dx$ on the Borel subsets of $[0,1)$ and let us consider a special type of non trivial couplings. We say that an absolutely continuous coupling $\pi$ of $dx$ with itself is a dyadic coupling if $K(x,y)=\tfrac{d\pi}{dx dy}=\varphi(\delta(x,y))$ for some nonnegative real function defined on $[0,1]$. Notice that $K(x,y)=\varphi(\delta(x,y))$ is a symmetric kernel since $\delta$ is symmetric. On the other hand since $\pi$ is a coupling probability measure, Lemma~\ref{lem:couplingMarkov} implies that $\int_{[0,1)}\varphi(\delta(x,y)) dy=1$ for every $x\in [0,1)$. Let us mention at this point that this type of kernels have been considered in \cite{AiGoMo18CLT} regarding some extension of the Central Limit Theorem. Set $\mathcal{M}_\delta (dx)$ to denote the class of the all kernels $K$ of the form $\varphi\circ \delta$ with $\int_{[0,1)}K(x,y)dy=1$.

The next statement contains some basic properties of these Markov kernel. Its proof can be found in \cite{AiGoMo18CLT}.
\begin{proposition}
	\quad
\begin{enumerate}[(1)]
	\item $K\in\mathcal{M}_{\delta}(dx)$ if and only if $K(x,y)=\sum_{j\geq 0}\alpha_j 2^j \mathcal{X}_{(0,2^{-j}]}(\delta(x,y))$ with
	\begin{enumerate}[(a)]
		\item $\sum_{j\geq 0} \abs{\alpha_j}<\infty$,
		\item $\sum_{l\leq j}\alpha_l 2^l\geq 0$, for every $j\geq 0$,
		\item $\sum_{j\geq 0} \alpha_j = 1$.
	\end{enumerate}
\item For $K\in\mathcal{M}_\delta (dx)$ the operator $\mathcal{K}f(x)=\int_{[0,1)}K(x,y)f(y)dy=\int_{[0,1)}\varphi(\delta(x,y)) dy$ has the special resolution $\mathcal{K}(\mathcal{X}_{[0,1)})=\mathcal{X}_{[0,1)]}$ and $\mathcal{K}h=\lambda_h h$, for $h\in\mathscr{H}$, with $\lambda_h=\sum_{j\geq j(h)}\alpha_j$, where $j(h)$ is the scale parameter of the support of $h$.
\item Since each $\lambda_h$ depends only on the scale of $j$ of $h$ we denote the sequence by $\{\lambda_j: j\geq 0\}$. With this notation $\lambda_0=1$ and $\lambda_j\to 0$ as $j\to\infty$.
\end{enumerate}
\end{proposition}

From the above proposition we readily obtain the spectral analysis for the Laplacian $\Delta_\pi$. As a consequence we obtain the Haar-Fourier approach to the solution of diffusion in this setting.

\begin{proposition}\label{thm:heatsolutioncoupling}
	Let $K\in\mathcal{M}_\delta(dx)$ as before. Then
	\begin{enumerate}
		\item $\Delta_\pi\mathcal{X}_{[0,1)}=0$ and $\Delta_\pi h= (\lambda_h -1)h$, $h\in\mathscr{H}$;
		\item for $f\in L^2([0,1), dx)$, the function $u(x,t)=\int_{[0,1)} f + \sum_{h\in\mathscr{H}}e^{t(\lambda_h -1)}\proin{f}{h} h(x)$ solves the problem
\begin{equation*}
(P)\,\left\{
\begin{array}{ll}
\frac{\partial u}{\partial t} = \Delta_\pi u,\, & t>0, x\in [0,1)\\
u(x,0)=f(x) \, & x\in [0,1).
\end{array}
\right.
\end{equation*}
\end{enumerate}	
\end{proposition}

The formula provided by $(2)$ in Proposition~\ref{thm:heatsolutioncoupling}  shows that the steady state of this diffusion is the mean value of the initial condition, in agreement with the discrete case in Proposition~\ref{propo:limitMarkovmatrix}, when $m_j=\tfrac{1}{n}$ for every $j$.

\section{Deterministic coupling}
Let $X$ be a locally compact Hausdorff topological space. Let $T$ be a Borel measurable mapping on $X$. Let $\mu$ be a Borel probability on $X$. Set $G:X\to X\times X$ be given by $G(x)=(x,T(x))$. For $A$ a Borel set in $X\times X$, define $\pi(A)=\mu(G^{-1}(A))$. Hence, the first marginal $\pi^1$ of $\pi$ is $\mu$ and the second is $\nu(B)=\mu(T^{-1}(B))$. See \cite{Villanibook} for details.
\begin{theorem}
Let $X$, $T$ and $\pi$ as above. Then for $\Phi\in\mathscr{C}_c(X\times X)$ we have
\begin{equation*}
Kir_\pi \Phi(x) = \Phi(x,T(x)).
\end{equation*}
\end{theorem}
\begin{proof}
In order to apply Theorem~\ref{thm:RadonNikodymFirstMarginal}, we have to compute the first marginal $\pi^1_\Phi$ of $\pi_\Phi$. Notice first that since for every Borel set $A$ in $X\times X$,
\begin{equation*}
\iint_{X\times X}\mathcal{X}_A(x,y) d\pi(x,y) = \pi(A) = \mu(G^{-1}(A))=\int_X \mathcal{X}_A(x,T(x)) d\mu(x),
\end{equation*}
we also have the formula
$\iint_{X\times X}\sigma(x,y) d\pi(x,y) = \int_X \sigma(x,T(x)) d\mu(x)$
for simple functions $\sigma$ and also for bounded measurable functions. Hence
\begin{equation*}
\pi_\Phi(A)=\iint_A \Phi d\pi=\iint_{X\times X}\mathcal{X}_A \Phi d\pi = \int_X \mathcal{X}_A(x,T(x)) \Phi(x,T(x)) d\mu(x).
\end{equation*}
So that
\begin{equation*}
\pi^1_\Phi(B)=\pi_\Phi(B\times X)=\int_B \Phi(x,T(x)) d\mu(x).
\end{equation*}
Hence
\begin{equation*}
Kir_\pi \Phi(x)=\frac{d\pi^1_\Phi(x)}{d\mu}=\Phi(x,T(x)).
\end{equation*}
\end{proof}

For a function $f\in\mathscr{C}_c(X)$ the corresponding Laplacian operator is then given by
\begin{equation*}
\Delta_\pi f(x) = f(T(x))-f(x).
\end{equation*}
Or, in operational form
\begin{equation*}
\Delta_\pi = \tau - I,
\end{equation*}
where $\tau f = f\circ T$ and $I$ is the identity.

Let us now consider the diffusion problem (P) in our current situation of the determi\-nis\-tic transport of $\mu$ through $T$.
 A way to get the solution of (P) in this setting is provided by the explicit computation of $e^{t\Delta_\pi}=\sum_{k=0}^{\infty}\frac{t^k}{k!}\Delta^k_\pi$.

\begin{lemma}\label{lemma:formulaLaplacianSemigroupIterated}
Let $X$, $\mu$, $T$ and $\pi$ be as before, then
\begin{equation*}
e^{t\Delta_\pi}f = e^{-t}\sum_{l\geq 0}\frac{t^l}{l!} (f\circ T^l)
\end{equation*}
for every $f\in \mathscr{C}_c(X)$.
\end{lemma}
\begin{proof}
Since $f$ is bounded, we see that the series above is absolutely convergent and that the $L^\infty$-norm of the right hand side is bounded by the $L^\infty$-norm of $f$. Since
\begin{equation*}\label{eq:laplacianIterated}
\Delta^k_\pi f = \sum_{l=0}^k \binom{k}{l}(-1)^{k-l} f\circ T^l,
\end{equation*}
we have
\begin{align*}
e^{t\Delta_\pi} f &= \sum_{k\geq 0}\frac{t^k}{k!}\Delta_\pi^k f\\
&=  \sum_{k\geq 0}\frac{t^k}{k!}\sum_{l=0}^k \binom{k}{l}(-1)^{k-l} f\circ T^l\\
&= \sum_{l\geq 0}f\circ T^l\sum_{k\geq l}\frac{(-1)^{k-l}}{k!}\binom{k}{l} t^k\\
&= \sum_{l\geq 0}\frac{t^l}{l!}f\circ T^l\sum_{k\geq l}\frac{(-1)^{k-l}}{(k-l)!} t^{k-l}\\
&= e^{-t}\sum_{l\geq 0}\frac{t^l}{l!}f\circ T^l.
\end{align*}
\end{proof}
Hence in the case of the Laplacian provided by a deterministic coupling of measures, the solution of the initial problem for the equation with initial data $f$, has a really wide diversity of steady states depending on the dynamics induced by the iterated system $\{T^l: l\geq 0\}$. In the next examples we only aim to illustrate this fact. For the case of ergodic mappings $T$, where $\lim_{t\to\infty}u(x,t)$ is a mean value of the initial condition, we mention a preliminary result due to F.~J.~ Mart\'in-Reyes \cite{MartinReyespreprint}. Nevertheless, perhaps more interesting from the point of view of the steady state as a classifier of coupling and transports are some particular non-ergodic cases as those considered in the next particular cases.
\begin{proposition}
Let $X=[-\tfrac{1}{2},\tfrac{1}{2}]$, $d\mu_1=dx$, $T(x)= -x$. Hence for every continuous function $f$ defined on $[-\tfrac{1}{2},\tfrac{1}{2}]$, its even part $f_e$ is the steady state of $e^{t\Delta}f$. Precisely
\begin{equation*}
e^{t\Delta}f\to f_e = \frac{f\circ T + f}{2}, \quad t\to \infty,
\end{equation*}
uniformly on $[-\tfrac{1}{2},\tfrac{1}{2}]$.
\end{proposition}
\begin{proof}
Let us apply Lemma~\ref{lemma:formulaLaplacianSemigroupIterated} in the current setting. Note that $f\circ T^l(x)=f((-1)^lx)$. Hence, with $f=f_e+f_o$ and $f_o=\frac{f(x)-f(-x)}{2}$ the odd part of $f$,
\begin{align*}
e^{t\Delta}f(x) &= e^{-t}\sum_{l\geq 0}\frac{t^l}{l!} f((-1)^l x)\\
&= e^{-t}\sum_{l\geq 0}\frac{t^l}{l!}\left(f_e(x) + f_o((-1)^l x)\right)\\
&= e^{-t}\left(\sum_{l\geq 0}\frac{t^l}{l!}\right) f_e(x) + e^{-t}\left(\sum_{l\geq 0}\frac{(-t)^l}{l!}\right)f_o(x)\\
&= f_e(x) + e^{-2t} f_o(x).
\end{align*}
\end{proof}

The next result deals with the Cantor function in $[0,1]$.
\begin{proposition}
	Let $X=[0,1]$, $T$ the Cantor function and $\mu$ the Hausdorff probability measure supported in the Cantor set contained in $[0,1]$. Then $\pi=\mu\circ G^{-1}$, with $G(x)=(x,Tx)$ is a coupling between $\mu$ and $dx$ in $[0,1]$. Then, the steady state of the solution of (P) is given by the function $g=f\circ T$.
\end{proposition}
\begin{proof}
	Since $T$ is continuous, the uniform probability $\mu$ on the Cantor set is given on intervals by  $\mu([a,b))=T(b)-T(a)$. Let $\nu=\mu\circ T^{-1}$. For $[c,d]$ in $[0,1]$ with $c$ and $d$ that do not belong to the set of dyadic numbers $\{k2^{-j}: j\geq 0; k=0,\ldots,2^j-1\}$ we have that $T^{-1}(\{c\})$ and $T^{-1}(\{d\})$ are singletons, or $\alpha=T^{-1}(c)$, $\beta=T^{-1}(d)$. Then $\nu([c,d])=\mu(T^{-1}([c,d]))=\mu([T^{-1}(c),T^{-1}(d)])=T(T^{-1}(d))-T(T^{-1}(c))=d-c$. Since the complement of the dyadic numbers of $[0,1]$ is dense in $[0,1]$, we get that $d\nu=dx$. On the other hand, the first and second marginals of $\pi$ are $\mu$ and $\nu$ respectively. The solution of
	\begin{equation*}
	\left\{
	\begin{array}{ll}
	\frac{\partial u}{\partial t} = \Delta_\pi u,\, & x\in [0,1), t>0, \\
	u(x,0)=f(x) \, & x\in [0,1).
	\end{array}
	\right.
	\end{equation*}
with $f$ continuous on $[0,1]$ is given by $u(x,t)=e^{-t}\sum_{l\geq 0} \frac{t^l}{l!} f\circ T^l(x)$. Let us compute $T^l(x)$ for $x\in [0,1]$. Assume that $x\in L^j_k$ the $k$-th middle third deleted in the $j$-th approximation of the Cantor set. The central point of $L^j_k$ is $k2^{-j}$ and $T(k2^{-j})=k2^{-j}$. Hence $T(x)=k2^{-j}$, $T^2(x)=T(k2^{-j})=k2^{-j}$. So that $T^l(x)=x$ for $l=0$ and $T^l(x)=k2^{-j}$ for every $l\geq 1$. Thus
\begin{align*}
u(x,t) &= e^{-t}\left[f(x)+ \left(\sum_{l\geq 1}\frac{t^l}{l!}\right)f(k2^{-j})\right]\\
&= e^{-t} f(x) + e^{-t}(e^t-1) f(k2^{-j})\\
&= e^{-t} [f(x)-f(k2^{-j})] + f(k2^{-j}).
\end{align*}	
Which tends to $f(k2^{-j})=f(T(x))$ for $t\to\infty$.
\end{proof}

\providecommand{\bysame}{\leavevmode\hbox to3em{\hrulefill}\thinspace}
\providecommand{\MR}{\relax\ifhmode\unskip\space\fi MR }
\providecommand{\MRhref}[2]{%
	\href{http://www.ams.org/mathscinet-getitem?mr=#1}{#2}
}
\providecommand{\href}[2]{#2}



\bigskip

\bigskip
\noindent{\footnotesize
\textsc{Instituto de Matem\'{a}tica Aplicada del Litoral, UNL, CONICET.}

\smallskip
\noindent\textmd{CCT CONICET Santa Fe, Predio ``Alberto Cassano'', Colectora Ruta Nac.~168 km 0, Paraje El Pozo, S3007ABA Santa Fe, Argentina.}
}
\bigskip

\end{document}